\documentclass[reqno]{amsart}
\usepackage{amssymb}
\usepackage{enumerate}
\usepackage[mathscr]{euscript}
\usepackage[hidelinks]{hyperref}
\usepackage{cleveref}
\usepackage{mathtools}
\usepackage{tikz-cd}
\usepackage{changepage}

\numberwithin{equation}{section}

\newtheorem{prop}[equation]{Proposition}
\newtheorem{lem}[equation]{Lemma}

\theoremstyle{definition}
\newtheorem{defn}[equation]{Definition}
\newtheorem{ex}[equation]{Example}
\newtheorem{rec}[equation]{Recollection}
\newtheorem{rem}[equation]{Remark}

\newcommand{\nc}{\newcommand}
\nc{\dmo}{\DeclareMathOperator}

\nc{\bbZ}{\mathbb{Z}}
\nc{\bbQ}{\mathbb{Q}}

\nc{\mfp}{\mathfrak{p}}

\nc{\rmc}{\mathrm{c}}
\nc{\rmb}{\mathrm{b}}
\nc{\rmD}{\mathrm{D}}
\nc{\rmh}{\mathrm{h}}
\nc{\rmop}{\mathrm{op}}
\nc{\rmr}{\mathrm{r}}
\nc{\rmS}{\mathrm{S}}

\nc{\scB}{\mathscr{B}}
\nc{\scC}{\mathscr{C}}
\nc{\scK}{\mathscr{K}}
\nc{\scM}{\mathscr{M}}
\nc{\scS}{\mathscr{S}}
\nc{\scT}{\mathscr{T}}

\dmo{\Ab}{Ab}
\dmo{\Coh}{Coh}
\dmo{\colim}{colim}
\dmo{\Hom}{Hom}
\dmo{\im}{Im}
\dmo{\Ker}{Ker}
\dmo{\Mod}{Mod}
\dmo{\smod}{mod}
\dmo{\smodu}{\underline{mod}}
\dmo{\smodurel}{\underline{mod_{rel}}}
\dmo{\loc}{loc}
\dmo{\Spc}{Spc}
\dmo{\Spec}{Spec}
\dmo{\supp}{supp}
\dmo{\thick}{thick}
\dmo{\Thick}{Thick}

\nc{\hooklongrightarrow}{\lhook\joinrel\longrightarrow}
\nc{\ot}{\otimes}
\nc{\wh}{\widehat}
\nc{\xr}{\xrightarrow}

\begin{document}

\title[Non-distributive lattices of thick tensor-ideals]{Non-distributive lattices of thick tensor-ideals via trivial extensions}
\author{Charalampos Verasdanis}
\address{Institute of Mathematics, Czech Academy of Sciences, \v{Z}itn\'a 25, 115 67 Prague, Czech Republic}
\email{verasdanis@math.cas.cz}
\date{}
\thanks{This research was supported by the project LQ100192601 Lumina quaeruntur, funded by the Czech Academy of Sciences (RVO 67985840).}

\begin{abstract}
\begin{adjustwidth}{0cm}{-0.1cm}
We construct non-rigid tensor-triangulated categories with non-distributive lattice of thick tensor-ideals.
\end{adjustwidth}
\end{abstract}

\maketitle

\section{Introduction}

Understanding the lattice, or certain sublattices, of thick subcategories of a triangulated category has been an ongoing endeavor since the pioneering works~\cite{BensonCarlsonRickard97,DevinatzHopkinsSmith88,Hopkins87,Neeman92,Thomason97}. This is a difficult problem that has been solved only in specific contexts. Notably, the lattice of radical thick tensor-ideals of a tensor-triangulated category is a spatial frame: isomorphic to the lattice of open subsets of a suitable space~\cite{Balmer05,BuanKrauseSolberg07}.

A necessary condition for a lattice to be a spatial frame is distributivity and it has been realized by Gratz--Stevenson~\cite{GratzStevenson23} that for such algebraic lattices, distributivity is in fact sufficient. Distributivity of the lattice of thick subcategories also often reflects properties of the ambient category; see~\cite{GratzStevenson26,JiangStevenson26} for consequences in this direction.

If a tensor-triangulated category is rigid, then the lattice of all thick tensor-ideals is distributive. On the other hand, counterexamples illustrating the failure of distributivity for non-rigid categories do not seem to be known. We produce such counterexamples. This is achieved via the following construction:

Given a tensor-triangulated category $\scT$ that acts on a triangulated category $\scK$, in the sense of~\cite{Stevenson13}, we define the \emph{trivial extension} $\scT\ltimes \scK$, which resembles the classical notion in ring theory. This is a tensor-triangulated category that is non-rigid as soon as $\scK\neq 0$; see~\Cref{sec:trivial-extensions}. Trivial extensions introduce complications in the lattice of ideals that are not detectable by the spectrum of prime ideals and the homological spectrum does not provide additional information either; see~\Cref{sec:homological-spectrum} for the latter. The claimed examples appear in~\Cref{sec:examples}. In~\Cref{sec:cg}, we explore trivial extensions of compactly generated categories and show that the lattice of smashing tensor-ideals can fail distributivity in the non-rigid setting, as opposed to being a frame in the rigid setting~\cite{BalmerKrauseStevenson20}.

\section{Trivial extensions}\label{sec:trivial-extensions}

We refer the reader to~\cite{Balmer05,Stevenson13,Verasdanis26} for background on tensor-triangulated categories and tensor-actions and to~\cite{GratzStevenson23,Krause26,Stevenson26} for the relevant lattice theory.

Let $\scT=(\scT,\ot,1)$ be an essentially small tensor-triangulated category that acts on an essentially small triangulated category $\scK$ via a functor $-\ast-\colon \scT\times \scK \to \scK$. This is a triangulated functor in both variables that satisfies associativity and unitality: there are natural isomorphisms $(X\ot Y) \ast A \cong X\ast (Y\ast A)$ and $1\ast A\cong A$, for all $X,Y\in \scT$, for all $A\in \scK$. The product category $\scT\times \scK$ becomes a triangulated category in the obvious way. We define a functor $- \ot - \colon (\scT \times \scK) \times (\scT \times \scK) \to \scT\times \scK$ as follows:
\[
(X,A)\ot (Y,B)=(X\ot Y,X\ast B \oplus Y\ast A),\, \forall X,Y\in \scT,\, \forall A,B\in \scK.
\]
If $(f_1,g_1)\colon (X_1,A_1)\to (Y_1,B_1)$ and $(f_2,g_2)\colon (X_2,A_2)\to (Y_2,B_2)$ are two morphisms of $\scT\times \scK$, then $(f_1,g_1)\ot (f_2,g_2)=(f_1\ot f_2,f_1\ast g_2 \oplus f_2\ast g_1)$. It is a matter of tedious computations to verify that this defines a symmetric tensor structure on $\scT\times \scK$ with unit $(1,0)$ that renders $\scT\times \scK$ a tensor-triangulated category.

\begin{defn}\label[defn]{defn:trivial-extension}
Endowed with the tensor structure defined above, the tensor-triangulated category $\scT\times \scK$ is called the \emph{trivial extension} of $\scT$ by $\scK$, which we denote by $\scT\ltimes \scK$.
\end{defn}

Restriction along tensor-triangulated functors yields tensor-actions:

\begin{lem}\label[lem]{lem:actions-tt-functors}
Let $F\colon \scT_1\to \scT_2$ be a tensor-triangulated functor. Then $F$ induces an action of $\scT_1$ on $\scT_2$ by the formula $X\ast Y=F(X)\ot Y,\, \forall X\in \scT_1,\, \forall Y\in \scT_2$.
\end{lem}

\begin{proof}
Easily verifiable.
\end{proof}

\begin{rem}\label[rem]{rem:canonical-injection-projection}
The canonical injection $\scT\to \scT\ltimes \scK$ and the canonical projection $\scT\ltimes \scK\to \scT$ are tensor-triangulated functors. Thus, there is an equivalence $\scT\cong (\scT\ltimes \scK)/(0\times \scK)$ of tensor-triangulated categories. The injection $\scT\to \scT\ltimes \scK$ defines an action of $\scT$ on $\scT\ltimes \scK$. Specifically, $X\ast (Y,A)\coloneqq(X,0)\ot (Y,A)=(X\ot Y,X\ast A)$. With respect to this action, the canonical injection $\scK\to \scT\ltimes \scK$ and the canonical projection $\scT\ltimes \scK\to \scK$ are action-preserving triangulated functors. Hence, there is an action-preserving equivalence $\scK\cong(\scT\ltimes \scK)/(\scT\times 0)$ of triangulated categories.
\end{rem}

The category $\scT\ltimes \scK$ is non-rigid, except from the trivial case:

\begin{lem}\label[lem]{lem:trivial-extension-rigid}
The category $\scT\ltimes \scK$ is rigid if and only if $\scT$ is rigid and $\scK=0$.
\end{lem}

\begin{proof}
Let $(X,A)$ be an object of $\scT\ltimes \scK$. Then it is straightforward to verify that $(X,A)^{\ot n}=(X^{\ot n},X^{\ot n-1} \ast A^{\oplus n})$. Hence, $(X,A)$ is tensor-nilpotent if and only if $X$ is tensor-nilpotent. Therefore, the subcategory of tensor-nilpotent objects of $\scT\ltimes \scK$ is $\mathrm{Nil}(\scT\ltimes \scK)=\mathrm{Nil}(\scT)\times \scK$. If $\scT\ltimes \scK$ is rigid, then $\mathrm{Nil}(\scT\ltimes \scK)=0$, which implies that $\scK=0$ and $\scT\cong \scT\ltimes \scK$ is rigid. The converse is evident.
\end{proof}

Next, we describe the thick tensor-ideals of a trivial extension. A thick subcategory $N$ of $\scK$ is a called a \emph{thick submodule} if $\scT\ast N\subseteq N$. We denote by $\Thick^\ast(\scK)$ the set of thick submodules of $\scK$. Ordered by inclusion, $\Thick^\ast(\scK)$ is a complete lattice. Meets are given by intersections. The join of a family $\{N_i\}$ of thick submodules is given by $\thick^\ast(\bigcup N_i)=\bigcap \{N\in \Thick^\ast(\scK)\mid \bigcup N_i\subseteq N\}$. Setting $\scK=\scT$ and $\ast=\ot$ recovers the notions for tensor-ideals. We denote by $\Thick(-)$ the lattice of all thick subcategories.

\begin{lem}\label[lem]{lem:ideals-trivial-extension}
The thick tensor-ideals of $\scT\ltimes \scK$ are of the form $I\times N$, where $I$ is a thick tensor-ideal of $\scT$ and $N$ is a thick submodule of $\scK$ such that $I\ast \scK \subseteq N$.
\end{lem}

\begin{proof}
If $I\subseteq \scT$ and $N\subseteq \scK$ are thick subcategories, then $I\times N$ is a thick subcategory of $\scT\times \scK$. Let $J$ be a thick subcategory of $\scT\times \scK$ and consider the thick subcategories
\begin{align*}
J_\scT&=\{X\in \scT\mid \exists A\in \scK:(X,A)\in J\}\subseteq \scT,\\
J_\scK&=\{A\in \scK\mid \exists X\in \scT:(X,A)\in J\}\subseteq \scK.
\end{align*}
Thickness of $J$ implies $J=J_\scT\times J_\scK$. Thus, $\Thick(\scT\times \scK)=\Thick(\scT)\times \Thick(\scK)$. Suppose that $I$ is a tensor-ideal and $N$ is a submodule such that $I\ast \scK\subseteq N$. Let $(X,A)\in \scT\ltimes \scK$ and $(Y,B)\in I\times N$. Then $(X,A)\ot (Y,B)=(X\ot Y,X\ast B \oplus Y\ast A)\in I\times N$. Therefore, $I\times N$ is a thick tensor-ideal of $\scT\ltimes \scK$. Conversely, suppose that a thick subcategory $I\times N\subseteq \scT\ltimes \scK$ is a tensor-ideal. The inverse image of $I\times N$ under the injection $\scT\to \scT\ltimes \scK$ (resp.~$\scK\to \scT\ltimes \scK$) is $I$ (resp.~$N$). Hence, $I$ is a tensor-ideal and $N$ is a submodule. Let $X\in I$ and $A\in \scK$. Then $(X,0)\in I\times N$ and since $I\times N$ is a tensor-ideal of $\scT\ltimes \scK$, we have $(0,X\ast A)=(X,0)\ot (0,A)\in I\times N$, which implies that $X\ast A\in N$. So, $I\ast \scK\subseteq N$.
\end{proof}

\begin{prop}\label[prop]{prop:non-distributive}
The lattice $\Thick^\ot(\scT\ltimes \scK)$ is distributive if and only if the lattices $\Thick^\ot(\scT)$ and $\Thick^\ast(\scK)$ are distributive.
\end{prop}

\begin{proof}
The claim is a consequence of the lattice embeddings
\begin{align*}
\Thick^\ot(\scT) &\hooklongrightarrow \Thick^\ot(\scT\ltimes \scK)\\
\mathmakebox[\widthof{$\Thick^\ot(\scT)$}]{I} &\longmapsto \mathmakebox[\widthof{$\Thick^\ot(\scT\ltimes \scK)$}]{I\times \scK}\\[4pt]
\Thick^\ast(\scK) &\hooklongrightarrow \Thick^\ot(\scT\ltimes \scK)\\
\mathmakebox[\widthof{$\Thick^\ast(\scK)$}]{N} &\longmapsto \mathmakebox[\widthof{$\Thick^\ot(\scT\ltimes \scK)$}]{0\times N}\\[4pt]
\Thick^\ot(\scT\ltimes \scK)&\hooklongrightarrow \Thick^\ot(\scT)\times \Thick^\ast (\scK)\\
\mathmakebox[\widthof{$\Thick^\ot(\scT\ltimes \scK)$}]{I\times N} &\longmapsto \mathmakebox[\widthof{$\Thick^\ot(\scT)\times \Thick^\ast(\scK)$}]{I\times N}
\end{align*}
that exist by~\Cref{rem:canonical-injection-projection} and~\Cref{lem:ideals-trivial-extension}.
\end{proof}

Now we show that a trivial extension $\scT\ltimes \scK$ has the same spectrum as $\scT$. We denote by $\Spc(-)$ the spectrum of prime thick tensor-ideals and by $\Thick^{\sqrt{\ot}}(-)$ the lattice of radical thick tensor-ideals.

\begin{lem}\label[lem]{lem:prime-ideals-trivial-extension}
The prime thick tensor-ideals of $\scT\ltimes \scK$ are of the form $\mfp\times \scK$, where $\mfp$ is a prime thick tensor-ideal of $\scT$. There is a homeomorphism $\Spc(\scT\ltimes \scK)\cong \Spc(\scT)$ and a lattice isomorphism $\Thick^{\sqrt{\ot}}(\scT\ltimes \scK)\cong \Thick^{\sqrt{\ot}}(\scT)$. The radical thick tensor-ideals of $\scT\ltimes \scK$ are of the form $I\times \scK$, where $I$ is a radical thick tensor-ideal of $\scT$.
\end{lem}

\begin{proof}
Let $\mfp\times N$ be a prime thick tensor-ideal of $\scT\ltimes \scK$. Then $\mfp$ must be a prime ideal of $\scT$, since it is the inverse image of $\mfp\times N$ under the injection $\scT\to\scT\ltimes \scK$. Let $A\in \scK$. Then $(0,A)\ot (0,A)=(0,0)\in \mfp\times N$, hence $(0,A)\in \mfp\times N$, i.e., $A\in N$. This shows that $N=\scK$. If $\mfp$ is a prime ideal of $\scT$, then the ideal $\mfp\times \scK$ must be prime, since it is the inverse image of $\mfp$ under the projection $\scT\ltimes \scK\to \scT$. We have essentially shown that the continuous maps between $\Spc(\scT\ltimes \scK)$ and $\Spc(\scT)$ induced by the functors $\scT\to \scT\ltimes \scK$ and $\scT\ltimes \scK\to \scT$ exhibit a homeomorphism. The claimed lattice isomorphism between the lattices of radical thick tensor-ideals and their explicit form follows from Balmer's classification theorem~\cite{Balmer05}.
\end{proof}

\begin{rem}\label[rem]{rem:unit-generates}
If $\scT=\thick(1)$, then every thick subcategory of $\scT$ is a tensor-ideal and every thick subcategory of $\scK$ is a submodule. In this case, the lattice of thick tensor-ideals of $\scT\ltimes \scK$ consists of all thick subcategories of the form $I\times N$ with $I\ast \scK \subseteq N$. Note that if $\scK\neq 0$, then it follows from~\Cref{lem:ideals-trivial-extension} that there is always a thick subcategory of $\scT\ltimes \scK$, namely $\scT\times 0$, that is not a tensor-ideal.
\end{rem}

\begin{rem}\label[rem]{rem:derived}
Even though trivial extensions, as in~\Cref{defn:trivial-extension}, resemble and share similarities with trivial extensions of rings, there is no obvious connection between the two concepts. For instance, the derived category of perfect complexes $\rmD^\mathrm{perf}(R\ltimes M)$, where $R$ is a commutative ring, $M$ is an $R$-module and $R\ltimes M$ is the trivial extension of $R$ by $M$, is a rigid tensor-triangulated category. Therefore, $\rmD^\mathrm{perf}(R\ltimes M)$ cannot be tensor-triangular equivalent to any trivial extension $\scT\ltimes \scK$, with $\scK\neq 0$, as the latter is non-rigid by~\Cref{lem:trivial-extension-rigid}.
\end{rem}

\section{Examples}\label{sec:examples}

First, we present three examples of non-rigid tensor-triangulated categories, where in each case the lattice of thick tensor-ideals is distributive.

\begin{ex}\label[ex]{ex:non-rigid-distributive}
Let $R$ be a commutative noetherian ring.
\begin{enumerate}[\rm(a)]
\item
The derived category $\rmD^{-}(\smod R)$ of bounded above complexes of finitely generated modules is a non-rigid tensor-triangulated category. If $R$ is artinian, then every thick tensor-ideal is radical. In this case, the lattice $\Thick^\ot(\rmD^{-}(\smod R))$ is isomorphic to the lattice of specialization closed subsets of $\Spec R$; see~\cite[Theorem 6.5]{MatsuiTakahashi17}.
\item
Let $Q$ be a finite acyclic quiver. The derived category $\rmD^\rmc(RQ)$ of the path algebra $RQ$, with vertex-wise tensor product, is a non-rigid tensor-triangulated category. Every thick tensor-ideal is radical and the lattice $\Thick^\ot(\rmD^\rmc(RQ))$ is isomorphic to the lattice of specialization closed subsets of $\Spc(\rmD^\rmc(RQ))=(\Spec R) \times Q_0$; see~\cite{Sabatini25}.
\item
Let $\scC$ be a finite EI-category and let $k$ be a field. The derived category $\rmD^\rmb(\smod k\scC)$ of bounded complexes of finite-dimensional modules over the category algebra $k\scC$ is a non-rigid tensor-triangulated category. Every thick tensor-ideal is radical and its spectrum is noetherian; see~\cite{Xu14}.
\end{enumerate}
\end{ex}

The next example is vastly more complicated but worthy of future exploration.

\begin{ex}\label[ex]{ex:relative-stable-modules}
Let $R$ be a commutative ring and let $G$ be a finite group. Let $\smod RG$ be the category of finitely generated $RG$-modules. The collection of short exact sequences in $\smod RG$ that split over $R$ is a Frobenius exact structure. The stable category of $\smod RG$ with respect to this exact structure is called the relative stable module category and is denoted by $\smodurel RG$~\cite{Broue09}. When $R$ is a field, $\smodurel RG$ coincides with the stable module category $\smodu RG$. The category $\smodurel RG$ is tensor-triangulated by tensoring modules over $R$. In general, $\smodurel RG$ is non-rigid. For example, let $S$ be a discrete valuation ring with residue field $k$ and uniformizing parameter $t\in S$. The main result of~\cite{BalandChirvasituStevenson19} is that the spectrum of $\smodurel (S/t^n)G$ decomposes as $\Spc(\smodurel (S/t^n)G)\cong \coprod_{i=1}^n \Spc(\smodu kG)$. However, the tensor-unit $S/t^n$ is indecomposable and hence $\smodurel (S/t^n)G$ cannot be rigid; see~\cite[Appendix A]{BensonIyengarKrause13}. See also~\cite[Section 3.2.4]{Stevenson26} for more explanations. From the decomposition of its spectrum, one can obtain a complete description of the lattice of radical thick tensor-ideals of $\smodurel (S/t^n)G$. The structure of the lattice of all thick tensor-ideals of $\smodurel (S/t^n)G$ is unclear.
\end{ex}

Using trivial extensions, we produce examples of non-distributive lattices of thick tensor-ideals. They all fit the same pattern.

\begin{ex}\label[ex]{ex:bounded-derived-categories}
Let $k$ be a field and let $\scT=\rmD^\rmb(\smod k)$. Let $\scK$ be either one of the derived categories $\rmD^\rmb(\smod kG)$, $\rmD^\mathrm{perf}(kG)$ or the stable module category $\smodu kG$ (which is equivalent to $\rmD^\rmb(\smod kG)/\rmD^\mathrm{perf}(kG)$) where $G$ is a finite group. Then $\scT$ acts on $\scK$, in each case via the appropriate base change functor induced from the canonical ring homomorphism $k\to kG$, i.e., via tensoring over $k$. Since $\scT=\thick(k)$, every thick subcategory of $\scT$ is a tensor-ideal and every thick subcategory of $\scK$ is a submodule. Therefore, by~\Cref{prop:non-distributive}, as soon as $\Thick(\scK)$ is non-distributive (and there are plenty such examples) $\Thick^\ot(\scT\ltimes \scK)$ is also non-distributive. By~\cite{GratzStevenson26}, for $k$ an algebraically closed field of characteristic $p$ dividing the order of $G$, we know that $\Thick(\rmD^\rmb(\smod kG))$ is distributive if and only if $G$ is $p$-nilpotent. Further, if $k$ is a field of characteristic $3$ and $G=C_3\times S_3$, then $\Thick(\smodu kG)$ is non-distributive.
\end{ex}

\begin{ex}\label[ex]{ex:ideals-Dynkin}
Let $k$ be a field and $kA_2$ the path algebra of the $A_2$ Dynkin quiver. Tensoring over $k$ defines an action of $\rmD^\rmb(\smod k)$ on $\rmD^\rmb(\smod kA_2)$. Let $\scC$ be the trivial extension $\rmD^\rmb(\smod k)\ltimes \rmD^\rmb(\smod kA_2)$. The lattice of thick subcategories of $\rmD^\rmb(\smod kA_2)$ is the diamond non-distributive lattice. It follows by~\Cref{prop:non-distributive} that the lattice $\Thick^\ot(\scC)$ is non-distributive. By using~\Cref{lem:ideals-trivial-extension} and~\Cref{lem:prime-ideals-trivial-extension} and the fact that $\Thick(\rmD^\rmb(\smod k))$ consists only of the trivial thick subcategories, we can easily compute the lattices of thick tensor-ideals and radical thick tensor-ideals of $\scC$:
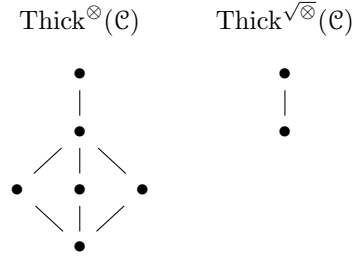
\[
\begin{tikzcd}[row sep=1em, column sep=1em]
&\mathmakebox[\widthof{$\bullet$}]{\Thick^\ot(\scC)} &&&&& \mathmakebox[\widthof{$\bullet$}]{\Thick^{\sqrt{\ot}}(\scC)}
\\[-0.5em]
&\bullet \dar[dash] &&&&& \bullet \dar[dash]
\\
&\bullet \dar[dash] \drar[dash] \dlar[dash] &&&&& \bullet
\\
\bullet \drar[dash] &\bullet \dar[dash] &\bullet \dlar[dash]
\\
&\bullet
\end{tikzcd}
\]
\end{ex}

\begin{ex}\label[ex]{ex:projective-line}
Let $\mathbb{P}^1_k$ be the projective line over a field $k$ and consider the bounded derived category of coherent sheaves $\rmD^\rmb(\Coh \mathbb{P}^1_k)$. The derived inverse image functor $\rmD^\rmb(\smod k)\to \rmD^\rmb(\Coh \mathbb{P}^1_k)$ induced by the canonical morphism $\mathbb{P}^1_k\to \Spec(k)$ is a tensor triangulated functor and therefore defines an action of $\rmD^\rmb(\smod k)$ on $\rmD^\rmb(\Coh \mathbb{P}^1_k)$. Let $\scC$ be the trivial extension $\rmD^\rmb(\smod k)\ltimes \rmD^\rmb(\Coh \mathbb{P}^1_k)$. The lattice of thick subcategories of $\rmD^\rmb(\Coh \mathbb{P}^1_k)$ is non-distributive, since it contains a copy of the pentagon non-distributive lattice; see~\cite[Example 4.5.1]{GratzStevenson23}. By~\Cref{prop:non-distributive}, it follows that the lattice $\Thick^\ot(\scC)$ is non-distributive.
\end{ex}

\begin{ex}\label[ex]{ex:non-distributive-dbmod}
Let $k$ be a field and let $R=k[x,y]/(x,y)^2$. The lattice $\Thick(\rmD^\rmb(\smod R))$ is non-distributive; see~\cite[Example 8.3]{GratzStevenson26}. The tensor product of complexes defines an action of $\rmD^\mathrm{perf}(R)$ on $\rmD^\rmb(\smod R)$. By~\Cref{prop:non-distributive}, $\Thick^\ot(\rmD^\mathrm{perf}(R)\ltimes \rmD^\rmb(\smod R))$ is non-distributive.
\end{ex}

Trivial extensions can be used to construct lattices of tensor-ideals combinatorially. We illustrate this with two examples: free distributive lattices and chains.

\begin{ex}
The tensor product of $\scT$ is an action of $\scT$ on itself. By ~\Cref{lem:ideals-trivial-extension}, the lattice of thick tensor-ideals of $\scT\ltimes \scT$ is $\Thick^\ot(\scT\ltimes \scT)=\{I\times J\mid I,J\in \Thick^\ot(\scT):I\subseteq J\}$. We define a sequence of tensor-triangulated categories by iterating this process. Let $\scT_0=\scT$ and set $\scT_n=\scT_{n-1}\ltimes \scT_{n-1},\, \forall n\geq 1$. Let $\scT_0=\rmD^\rmb(\smod k)$. By the formula above, it follows that the lattice of thick tensor-ideals $\Thick^\ot(\scT_n)$ is precisely the free distributive lattice $FD(n)$ on $n$ generators, since it is the lattice of intervals of $\Thick^\ot(\scT_{n-1})$ and $\Thick^\ot(\scT_0)$ is a chain consisting of two elements. The number of elements of $FD(n)$ is given by the Dedekind number $M(n)$. Below, we depict the lattices $FD(0)$, $FD(1)$, $FD(2)$ and $FD(3)$. They have $M(0)=2$, $M(1)=3$, $M(2)=6$ and $M(3)=20$ elements, respectively. We do not attempt to graph the lattice $FD(4)$, as it is slightly more complicated with $168$ vertices and $504$ arrows. These lattices exhibit rapid combinatorial explosion and the number $M(n)$ has only been computed up to $n=9$; see~\cite{Jakel23,VHCGKRMP24}.
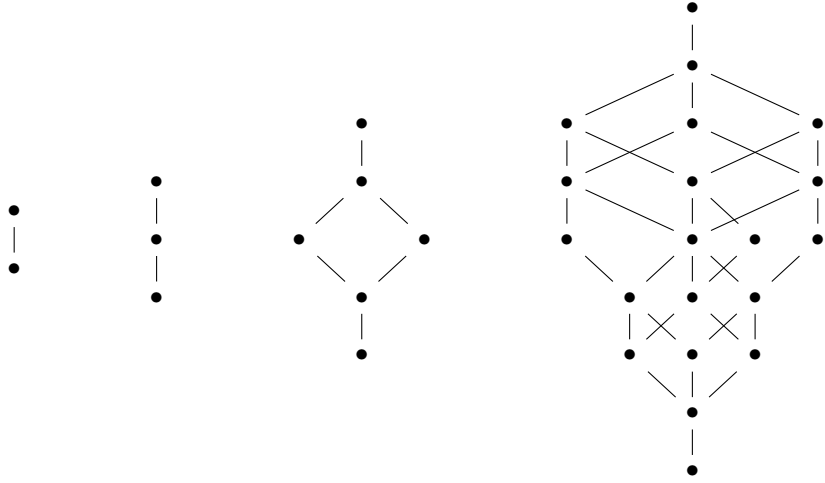
\[
\begin{tikzcd}[row sep=1em]
\bullet \dar[dash]
\\
\bullet
\end{tikzcd}
\qquad \qquad
\begin{tikzcd}[row sep=1em]
\bullet \dar[dash]
\\
\bullet \dar[dash]
\\
\bullet
\end{tikzcd}
\qquad \qquad
\begin{tikzcd}[row sep=1em,column sep=1em]
&\bullet \dar[dash]
\\
&\bullet \drar[dash] \dlar[dash]
\\
\bullet \drar[dash]&& \bullet \dlar[dash]
\\
&\bullet \dar[dash]
\\
&\bullet
\end{tikzcd}
\qquad \qquad
\begin{tikzcd}[row sep=1em,column sep=1em]
&& \bullet \dar[dash]
\\
&&\bullet \dar[dash] \arrow[dash]{dll} \arrow[dash]{drr}
\\
\bullet \dar[dash] \arrow[dash]{drr} && \bullet \arrow[dash]{dll} \arrow[dash]{drr} && \bullet \dar[dash] \arrow[dash]{dll}
\\
\bullet \dar[dash] \arrow[dash]{drr} && \bullet \dar[dash] \arrow[dash]{dr} && \bullet \arrow[dash]{dll} \dar[dash]
\\
\bullet \drar[dash] && \bullet \dlar[dash] \dar[dash] \drar[dash] & \bullet \dlar[dash] & \bullet \dlar[dash]
\\
&\bullet \drar[dash] \dar[dash] &\bullet \dlar[dash] \drar[dash] &\bullet \dlar[dash] \dar[dash]
\\
&\bullet \drar[dash] &\bullet \dar[dash] &\bullet \dlar[dash]
\\
&&\bullet \dar[dash]
\\
&&\bullet
\end{tikzcd}
\]
\end{ex}

\begin{ex}
The canonical injection $\scT\to \scT\ltimes \scT$ on the first component induces an action of $\scT$ on $\scT\ltimes \scT$. Take the trivial extension $\scT\ltimes (\scT\ltimes \scT)$. Iterate this process to obtain the sequence $\scT_n=\scT\ltimes \scT_{n-1},\, \forall n\geq 1$, where $\scT_0=\scT$. For $\scT=\rmD^\rmb(\smod k)$, the lattice of thick tensor-ideals of $\scT_n$ is the $(n+2)$-chain.
\end{ex}

\section{The homological spectrum of a trivial extension}\label{sec:homological-spectrum}
In this section, we show that the homological spectrum of a trivial extension $\scT\ltimes \scK$ is homeomorphic to the homological spectrum of $\scT$. We start by recalling some facts about module categories and their tensor structure and refer the reader to~\cite{BalmerKrauseStevenson20} for more in depth explanations. 

\begin{rec}\label[rec]{rec:module-categories}
Let $\scT$ be an essentially small tensor-triangulated category. We denote by $\Mod \scT$ the abelian category of additive functors $M\colon \scT^\rmop \to \Ab$ and by $\smod \scT$ its abelian subcategory of finitely presented objects.
The category $\Mod \scT$ receives $\scT$ via the Yoneda embedding
\[
\wh{(-)}\colon \scT\to \Mod \scT,\quad X\mapsto \Hom_\scT(-,X)
\]
and inherits a right exact colimit-preserving symmetric tensor via Day convolution, that makes the Yoneda embedding a tensor functor and $\smod \scT$ a tensor subcategory, in the following way: Let $M_1,M_2\in \Mod \scT$ and express them as colimits
\[
M_1=\colim_{\wh{X}\to M_1}\wh{X},\quad M_2=\colim_{\wh{Y}\to M_2} \wh{Y}
\]
over their respective categories of elements. Then
\[
M_1\ot M_2=\colim_{\wh{X}\to M_1}\colim_{\wh{Y}\to M_2} \wh{X\ot Y}.
\]
\end{rec}

Let $\scK$ be an essentially small triangulated category on which $\scT$ acts. Then $\Mod \scK$ inherits an action of $\Mod \scT$ via the analogous formula that defines the tensor product on $\Mod \scT$. Let
\[
M=\colim_{\wh{X}\to M}\wh{X}\in \Mod \scT, \quad N=\colim_{\wh{A}\to N}\wh{A}\in \Mod \scK
\]
and define
\[
M\ast N=\colim_{\wh{X}\to M}\colim_{\wh{A}\to N}\wh{X\ast A}.
\]
The functor $-\ast-\colon \Mod \scT\times \Mod \scK\to \Mod \scK$ is right exact and colimit-preserving in each variable and satisfies associativity: $(M_1\ot M_2) \ast N \cong M_1 \ast (M_2 \ast N)$ and unitality: $1\ast N\cong N,\, \forall M_1,M_2\in \Mod \scT,\, \forall N\in \Mod \scK$. Further, if $M\in \smod \scT$ and $N\in \smod \scK$, then $M\ast N\in \smod \scK$.

We define the \emph{trivial extension} $\Mod \scT\ltimes \Mod \scK$ as the product category $\Mod \scT\times \Mod \scK$ equipped with the tensor product defined by
\[
(M_1,N_1)\ot (M_2,N_2)=(M_1 \ot M_2, M_1\ast N_2 \oplus M_2\ast N_1),
\]
which is right exact and colimit-preserving in each variable and restricts on subcategories of finitely presented objects:
\[
-\ot - \colon (\smod \scT\times \smod \scK) \times (\smod \scT\times \smod \scK) \to \smod \scT\times \smod \scK.
\]

\begin{lem}\label[lem]{lem:module-categories-equivalence}
There is an equivalence $\Mod (\scT\ltimes \scK)\xr{\simeq} \Mod \scT\ltimes \Mod \scK$ of tensor abelian categories that restricts to an equivalence $\smod(\scT\ltimes \scK) \xr{\simeq} \smod \scT \ltimes \smod \scK$.
\end{lem}

\begin{proof}
The equivalence between $\Mod(\scT\times \scK)$ and $\Mod \scT\times \Mod \scK$ is a standard fact: Define the functor $\scT\times \scK\to \Mod \scT \times \Mod \scK$ that sends an object $(X,A)\in \scT\times \scK$ to $(\wh{X},\wh{A})\in \Mod \scT\times \Mod \scK$, i.e., the product of the Yoneda embeddings. This is a coproduct-preserving homological functor, so by the universal property of the Yoneda embedding $\scT\times \scK\to \Mod(\scT\times \scK)$, there is a unique exact colimit-preserving functor $F\colon \Mod(\scT\times \scK)\to \Mod \scT \times \Mod \scK$ that makes the following triangle commute:
\[
\begin{tikzcd}
\scT\times \scK \rar[hook] \drar[hook] & \Mod(\scT\times \scK) \dar[dashed,"\exists!"',"F"]
\\
& \Mod \scT \times \Mod \scK.
\end{tikzcd}
\]
The functor $F$ is an equivalence. To verify that $F$ preserves the trivial extension tensor structures, check the required property on objects of the form $\wh{(X,A)}\in \Mod(\scT\times \scK)$, i.e., objects in the image of the Yoneda embedding, and use the fact that $F$ preserves colimits. The fact that $F$ restricts to subcategories of finitely presented objects is clear.
\end{proof}

A Serre subcategory $\scC$ of $\smod \scK$ is called a \emph{Serre submodule} if $\smod \scT \ast \scC\subseteq \scC$.

\begin{lem}\label[lem]{lem:Serre-ideals}
The Serre tensor-ideals of $\smod \scT\ltimes \smod \scK$ are of the form $\scB\times \scC$, where $\scB$ is a Serre tensor-ideal of $\smod \scT$ and $\scC$ is a Serre submodule of $\smod \scK$ such that $\scB\ast \smod \scK \subseteq \scC$.
\end{lem}

\begin{proof}
The proof follows the same logic as that of~\Cref{lem:ideals-trivial-extension}. The Serre subcategories of $\smod \scT\times \smod \scK$ are of the form $\scB\times \scC$, where $\scB$ is a Serre subcategory of $\smod \scT$ and $\scC$ is a Serre subcategory of $\smod \scK$. The canonical injection $\smod \scT\to \smod \scT \ltimes \smod \scK$ is an exact tensor functor, hence $\scB$ is a Serre tensor-ideal of $\smod \scT$. The canonical injection $\smod \scK\to \smod \scT\ltimes \smod \scK$ is an action-preserving exact functor (where the action of $\smod \scT$ on $\smod \scT\ltimes \smod \scK$ is defined via restriction along the canonical injection) hence $\scC$ is a Serre submodule of $\smod \scK$. Using the fact that $\scB\times \scC$ is a tensor-ideal of $\smod \scT\ltimes \smod \scK$, one deduces that $\scB\ast \smod \scK\subseteq \scC$. Conversely, the condition $\scB \ast \smod \scK\subseteq \scC$ ensures that the Serre subcategory $\scB\times \scC$ is a tensor-ideal.
\end{proof}

\begin{rec}\label{rec:homological-spectrum}
The homological spectrum of $\scT$, denoted by $\Spc^\rmh(\scT)$, is the set of proper maximal Serre tensor-ideals of $\smod \scT$ equipped with the topology with basis of closed subsets those of the form $\supp^\rmh(X)$, where $X\in \scT$ and $\supp^\rmh$ is the homological support; see~\cite{Balmer20a,Balmer20b}.
\end{rec}

\begin{prop}\label[prop]{prop:homological-spectrum-trivial-extension}
There is a homeomorphism $\Spc^\rmh(\scT\ltimes \scK)\cong \Spc^\rmh(\scT)$.
\end{prop}

\begin{proof}
Let $\scB\times \scC$ be a proper maximal Serre tensor-ideal of $\smod \scT\ltimes \smod \scK$. Then $\scB$ is a proper Serre tensor-ideal of $\smod \scT$. Otherwise, $\scB=\smod \scT$ combined with the condition $\scB\ast \smod \scK\subseteq \scC$ implies that $\scC=\smod \scK$, contradicting the assumption that $\scB\times \scC$ is proper. Further, $\scB$ is maximal. Otherwise, there would be an inclusion $\scB\subsetneq \scB'$, implying that $\scB\times \scC\subsetneq \scB'\times \scC$, contradicting maximality of $\scB\times \scC$. Note that $\scB'\times \scC$ does not need to be a tensor-ideal a priori, since maximality of $\scB\times \scC$ as a Serre tensor-ideal implies its maximality as a Serre subcategory. Since $\scB$ is a maximal Serre tensor-ideal, it is prime with respect to the tensor product. It follows that $\scC=\smod \scK$ in the same way as in~\Cref{lem:prime-ideals-trivial-extension}. Any Serre subcategory of the form $\scB\times \smod \scK$ is a tensor-ideal and if $\scB$ is maximal, then so is $\scB\times \smod \scK$. We have shown that the proper maximal Serre tensor-ideals of $\smod \scT\ltimes \smod \scK$ are of the form $\scB\times \smod \scK$, where $\scB$ is a proper maximal Serre tensor-ideal of $\smod \scT$. The equivalence $\smod(\scT\ltimes \scK)\xr{\simeq} \smod \scT\ltimes \smod \scK$ established in~\Cref{lem:module-categories-equivalence} now sets up a bijection between $\Spc^\rmh(\scT\ltimes \scK)$ and $\Spc^\rmh(\scT)$, which is clearly a homeomorphism.
\end{proof}

\section{Compactly generated trivial extensions}\label{sec:cg}

Let $\scT$ be a compactly generated tensor-triangulated category. Specifically, we assume that the tensor product preserves all coproducts and that the subcategory of compact objects $\scT^\rmc$ is a tensor subcategory. We denote by $[-,-]$ the internal hom functor, which exists by Brown representability, and by $X^\vee=[X,1]$ the dual of an object $X\in \scT$. For the time being, we do not assume that $\scT^\rmc$ is necessarily rigid. Let $\scK$ be a compactly generated triangulated category on which $\scT$ acts. We assume that the action preserves all coproducts.

The trivial extension $\scT\ltimes \scK$ is a compactly generated tensor-triangulated category. Since its tensor unit $(1,0)$ is compact, every rigid object of $\scT\ltimes \scK$ is compact. In fact, it is rather easy to give a full description of the subcategory of rigid objects.

\begin{lem}\label[lem]{lem:rigid-objects-trivial-extension}
Let $(X,A)\in \scT\ltimes \scK$. Then $(X,A)^\vee\cong (X^\vee,0)$. The subcategory of rigid objects of $\scT\ltimes \scK$ is $\scT^\rmr\times 0$, where $\scT^r$ is the subcategory of rigid objects of $\scT$.
\end{lem}

\begin{proof}
Let $(X,A)$ and $(Y,B)\in \scT\ltimes \scK$. Then
\begin{align*}
\Hom_{\scT\ltimes \scK}((Y,B),(X,A)^\vee)&=\Hom_{\scT\ltimes \scK}((Y,B),[(X,A),(1,0)])\\
&=\Hom_{\scT\ltimes\scK}((Y,B)\ot (X,A),(1,0)) \\
&=\Hom_{\scT\ltimes \scK}((Y\ot X,Y\ast A\oplus X\ast B),(1,0))\\
&=\Hom_{\scT}(Y\ot X,1)\\
&=\Hom_{\scT}(Y,X^\vee)\\
&=\Hom_{\scT\ltimes \scK}((Y,B),(X^\vee,0)).
\end{align*}
Consequently, $(X,A)^\vee\cong (X^\vee,0)$. Since duality on the subcategory of rigid objects is an involution, it follows that if $(X,A)$ is rigid, then $A=0$ and $X$ must be rigid. We infer that the subcategory of rigid objects of $\scT\ltimes \scK$ is $\scT^\rmr\times 0$.
\end{proof}

Now we examine Brown--Comenetz duals. There are two possible variations: First, classically, the Brown--Comenetz dual of a compact object $X\in \scT$ is defined as the object $IX$ representing the functor $\Hom_\bbZ(\Hom_\scT(X,-),\bbQ/\bbZ)$. Taking advantage of the tensor structure, one can define the Brown--Comenetz dual of an arbitrary object $X\in \scT$ as the object $X^\dagger$ representing the functor $\Hom_\bbZ(\Hom_\scT(1,X\ot -),\bbQ/\bbZ)$. If $X$ is rigid (therefore $X$ and $X^\vee$ are necessarily compact) then $X^\dagger\cong IX^\vee$ and $IX\cong (X^\vee)^\dagger$, so the two notions coincide up to duality. By a similar computation as in~\Cref{lem:rigid-objects-trivial-extension}, we can find the Brown--Comenetz duals, with respect to the tensor structure, of objects in $\scT\ltimes \scK$.

\begin{lem}\label[lem]{lem:bc-duals}
Let $(X,A)$ be an object of $\scT\ltimes \scK$. Then $(X,A)^\dagger\cong (X^\dagger,0)$.
\end{lem}

\begin{proof}
Let $(Y,B)\in \scT\ltimes \scK$. Then
\begin{align*}
\Hom_{\scT\ltimes \scK}((Y,B),(X,A)^\dagger)&= \Hom_\bbZ(\Hom_{\scT\ltimes \scK}((1,0),(X,A)\ot (Y,B)),\bbQ/\bbZ)\\
&=\Hom_\bbZ(\Hom_{\scT\ltimes \scK}((1,0),(X\ot Y,X\ast B\oplus Y\ast A)),\bbQ/\bbZ)\\
&=\Hom_\bbZ(\Hom_\scT(1,X\ot Y),\bbQ/\bbZ)\\
&= \Hom_\scT(Y,X^\dagger)\\
&= \Hom_{\scT\ltimes \scK}((Y,B),(X^\dagger,0)).
\end{align*}
Conclusion: $(X,A)^\dagger\cong (X^\dagger,0)$.
\end{proof}

\begin{rem}\label[rem]{rem:bc-duals-compact-non-rigid}
If $(X,A)\in \scT\ltimes \scK$ is rigid, then by~\Cref{lem:rigid-objects-trivial-extension}, $A=0$ and $X$ is rigid. We have $(X,0)^\dagger=(X^\dagger,0)=(IX^\vee,I0)=I(X^\vee,0)=I(X,0)^\vee$, as expected. In case $(X,A)$ is compact but not rigid and $A\neq 0$, then $IA\neq 0$ and $I(X,A)=(IX,IA)$ but $((X,A)^\vee)^\dagger=((X^\vee)^\dagger,0)$.
\end{rem}

Next we discuss localizing and smashing tensor-ideals of $\scT\ltimes \scK$.

\begin{lem}\label[lem]{lem:localizing-ideals}
The localizing tensor-ideals of $\scT\ltimes \scK$ are of the form $L\times N$, where $L$ is a localizing tensor-ideal of $\scT$ and $N$ is a localizing submodule of $\scK$ such that $L\ast \scK\subseteq N$.
\end{lem}

\begin{proof}
The proof is exactly the same as that of~\Cref{lem:ideals-trivial-extension}.
\end{proof}

Recall that a \emph{smashing subcategory} $\scS$ of $\scT$ is a localizing subcategory such that the quotient functor $\scT\to \scT/\scS$ has a coproduct-preserving right adjoint. From now on, assume that $\scT$ is rigidly-compactly generated, i.e., $\scT^\rmc$ is rigid.

\begin{rec}\label[rec]{rec:smashing-ideals}
By~\cite{BalmerFavi11}, the smashing tensor-ideals of $\scT$ are completely determined by idempotent triangles $e\to 1\to f$, where $e\ot e\cong e$, $f\ot f\cong f$ and $e\ot f=0$ (these three conditions being equivalent). Specifically, if $\scS$ is a smashing tensor-ideal of $\scT$, then its associated left (resp.~right) idempotent $e_\scS$ (resp.~$f_\scS$) is the image of $1$ under the acyclization (resp.~localization) functor $\scT\to \scT$ and it holds that $\scS=\loc^\ot(e_\scS)=\im(e_\scS\ot -)=\Ker(f_\scS\ot -)$. The acyclization and localization functors are $\Gamma_\scS\cong \Gamma_\scS(1)\ot-=e_\scS\ot -$ and $L_\scS\cong L_\scS(1)\ot -=f_\scS\ot -$, respectively.
\end{rec}

\begin{rem}\label[rem]{rem:non-rigid-smashing}
In general, in a compactly generated tensor-triangulated category, the localizing tensor-ideal generated by a left idempotent is smashing. Rigidity is critical in establishing the converse. The issue is that if $\scS$ is a smashing tensor-ideal, the associated acyclization and localization functors $\Gamma_\scS$ and $L_\scS$ may not be isomorphic to $\Gamma_\scS(1)\ot -$ and $L_\scS(1)\ot -$ if rigidity is not assumed.
\end{rem}

\begin{prop}\label[prop]{prop:smashing-ideals-trivial-extension}
The smashing tensor-ideals of $\scT\ltimes \scK$ are of the form $\scS\times \scM$, where $\scS$ is a smashing tensor-ideal of $\scT$ and $\scM$ is a smashing submodule of $\scK$ such that $\scS\ast\scK \subseteq \scM$. The left idempotents of $\scT\ltimes \scK$ are of the form $(e,A)$, where $e$ is a left idempotent of $\scT$, say with structure map $\gamma\colon e\to 1$, and $e\ast A\xr{\gamma \ast A} A$ is an isomorphism. The smashing tensor ideals of $\scT\ltimes \scK$ generated by left idempotents stand in bijection with smashing tensor-ideals of $\scT$: They are of the form $\scS\times (\scS\ast \scK)$, where $\scS$ is a smashing tensor-ideal of $\scT$.
\end{prop}

\begin{proof}
The assertion about the form of smashing tensor-ideals of $\scT\ltimes \scK$ is a consequence of~\Cref{lem:localizing-ideals}, since the smashing subcategories of $\scT\times \scK$ are obtained as products of smashing subcategories of $\scT$ and $\scK$. Let $(e,A)\xr{(\gamma,0)} (1,0)$ be a left idempotent in $\scT\ltimes \scK$. Tensoring $(\gamma,0)$ with $(e,A)$ yields the isomorphism $(e\ot e,e\ast A \oplus e\ast A)\xr{(\gamma\ot e,(\gamma \ast A \, 0))} (e,A)$. It follows that $\gamma \ot e\colon e\ot e\to e$ and $\gamma \ast A\colon e\ast A\to A$ are isomorphisms. So, $\gamma \colon e\to 1$ is a left idempotent in $\scT$ and $A\in \im(e\ast -)=\loc^\ot(e)\ast \scK$. Hence, $\loc^\ot((e,A))\subseteq \loc^\ot(e)\times (\loc^\ot(e)\ast \scK)$. We claim that $\loc^\ot((e,0))=\loc^\ot(e)\times (\loc^\ot(e)\ast \scK)$. The inclusion $\loc^\ot((e,0))\subseteq \loc^\ot(e)\times (\loc^\ot(e)\ast \scK)$ is clear. Let $(X,B)\in \loc^\ot(e)\times (\loc^\ot(e)\ast \scK)$. Then $e\ot X\cong X$ and $e\ast B\cong B$. We have $(X,B)\cong (e\ot X,e\ast B)=(e,0)\ot (X,B)\in \loc^\ot((e,0))$, proving the claim. The ideal $\loc^\ot((e,A))$ must be of the form $L\times N$ and then $(e,A)\in L\times N$ implies $e\in L$. Consequently, $(e,0)\in L\times N$. We conclude that $\loc^\ot((e,A))=\loc^\ot((e,0))=\loc^\ot(e)\times (\loc^\ot(e)\ast \scK)$. All smashing tensor-ideals of $\scT$ are of the form $\loc^\ot(e)$, for some left idempotent $e$; see~\Cref{rec:smashing-ideals}. This completes the proof.
\end{proof}

\begin{rem}
Let $\scS$ be a smashing tensor-ideal of $\scT$ and let $\scM$ a smashing submodule of $\scK$ such that $\scS\ast \scK\subsetneq \scM$ (assuming that $\scK\neq 0$). By~\Cref{prop:smashing-ideals-trivial-extension}, it follows that $\scS\times \scM$ is a smashing tensor-ideal of $\scT\ltimes \scK$ that is not generated by a left idempotent. This already happens in the simplest case $\scS=0$ and $\scM=\scK$.
\end{rem}

By~\cite{BalmerKrauseStevenson20}, the lattice $\rmS^\ot(\scT)$ of smashing tensor-ideals of a rigidly-compactly generated tensor-triangulated category $\scT$ is a frame. We show that this can fail without the rigidity assumption. Let $\rmS^\ast(\scK)$ denote the lattice of smashing submodules of $\scK$.

\begin{rem}\label[rem]{rem:non-distr-smashing}
As in~\Cref{prop:non-distributive}, mapping a smashing submodule $\scM$ of $\scK$ to the smashing tensor-ideal $0\times \scM$ of $\scT\ltimes \scK$ provides a lattice embedding of $\rmS^\ast(\scK)$ into $\rmS^\ot(\scT\ltimes \scK)$. Therefore, if $\rmS^\ast(\scK)$ is non-distributive, then $\rmS^\ot(\scT\ltimes \scK)$ is also non-distributive.
\end{rem}

\begin{ex}\label[ex]{ex:non-distr-smashing}
Let $k$ be a field and consider the action of $\rmD(k)$ on $\rmD(\mathbb{P}^1_k)$ induced by the canonical morphism $\mathbb{P}^1_k\to \Spec(k)$. Since $\rmD(k)$ is generated by its tensor unit, every localizing subcategory of $\rmD(\mathbb{P}^1_k)$ is a submodule. The lattice of smashing subcategories of $\rmD(\mathbb{P}^1_k)$ is non-distributive; see~\cite[Remark 5.10]{BalmerKrauseStevenson20}. By~\Cref{rem:non-distr-smashing}, the lattice of smashing tensor-ideals of $\rmD(k)\ltimes \rmD(\mathbb{P}^1_k)$ is non-distributive.
\end{ex}


\begin{thebibliography}{99}
\bibitem{BalandChirvasituStevenson19}
S. Baland, A. Chirv\u asitu and G. Stevenson, The prime spectra of relative stable module categories, Trans. Amer. Math. Soc. {\bf 371} (2019), no.~1, 489--503.

\bibitem{Balmer05}
P. Balmer, The spectrum of prime ideals in tensor triangulated categories, J. Reine Angew. Math. {\bf 588} (2005), 149--168.

\bibitem{Balmer20a}
P. Balmer, Homological support of big objects in tensor-triangulated categories, J. \'Ec. polytech. Math. {\bf 7} (2020), 1069--1088.

\bibitem{Balmer20b}
P. Balmer, Nilpotence theorems via homological residue fields, Tunis. J. Math. {\bf 2} (2020), no.~2, 359--378.

\bibitem{BalmerFavi11}
P. Balmer and G. Favi, Generalized tensor idempotents and the telescope conjecture, Proc. Lond. Math. Soc. (3) {\bf 102} (2011), no.~6, 1161--1185.

\bibitem{BalmerKrauseStevenson20}
P. Balmer, H. Krause and G. Stevenson, The frame of smashing tensor-ideals, Math. Proc. Cambridge Philos. Soc. {\bf 168} (2020), no.~2, 323--343.

\bibitem{BensonCarlsonRickard97}
D.~J. Benson, J.~F. Carlson and J. Rickard, Thick subcategories of the stable module category, Fund. Math. {\bf 153} (1997), no.~1, 59--80.

\bibitem{BensonIyengarKrause13}
D.~J. Benson, S.~B. Iyengar and H. Krause, Module categories for group algebras over commutative rings, J. K-Theory {\bf 11} (2013), no.~2, 297--329,

\bibitem{Broue09}
M. Brou\'e, Higman's criterion revisited, Michigan Math. J. {\bf 58} (2009), no.~1, 125--179.

\bibitem{BuanKrauseSolberg07}
A.~B. Buan, H. Krause and \O. Solberg, Support varieties: an ideal approach, Homology Homotopy Appl. {\bf 9} (2007), no.~1, 45--74.

\bibitem{DevinatzHopkinsSmith88}
E.~S. Devinatz, M.~J. Hopkins and J.~H. Smith, Nilpotence and stable homotopy theory. I, Ann. of Math. (2) {\bf 128} (1988), no.~2, 207--241.

\bibitem{GratzStevenson23}
S. Gratz and G. Stevenson, Approximating triangulated categories by spaces, Adv. Math. {\bf 425} (2023), Paper No. 109073, 44 pp.

\bibitem{GratzStevenson26}
S. Gratz and G. Stevenson, Ore's theorem for thick subcategories, Preprint arXiv:2604.19123 (2026).

\bibitem{Hopkins87}
M.~J. Hopkins, Global methods in homotopy theory, in {\it Homotopy theory (Durham, 1985)}, 73--96, London Math. Soc. Lecture Note Ser., 117, Cambridge Univ. Press, Cambridge.

\bibitem{Jakel23}
C. J\"akel, A computation of the ninth Dedekind number, J. Comput. Algebra {\bf 6/7} (2023), Paper No. 100006, 8 pp.

\bibitem{JiangStevenson26}
A. Jiang and G. Stevenson, Distributivity, affineness, and the structure sheaf, Preprint arXiv:2604.18793 (2026).

\bibitem{Krause26}
H. Krause, Tensor triangular geometry -- Notes for an Oberwolfach Seminar, Preprint arXiv:2605.19983 (2026).

\bibitem{MatsuiTakahashi17}
H. Matsui and R. Takahashi, Thick tensor ideals of right bounded derived categories, Algebra Number Theory {\bf 11} (2017), no.~7, 1677--1738.

\bibitem{Neeman92}
A. Neeman, The chromatic tower for $D(R)$, Topology {\bf 31} (1992), no.~3, 519--532.

\bibitem{Sabatini25}
E. Sabatini, The Balmer spectrum and tensor telescope conjecture for noetherian path algebras, Preprint arXiv:2511.20204 (2025).

\bibitem{Stevenson13}
G. Stevenson, Support theory via actions of tensor triangulated categories, J. Reine Angew. Math. {\bf 681} (2013), 219--254.

\bibitem{Stevenson26}
G. Stevenson, Some notes on tensor triangular geometry, Preprint arXiv:2602.08480 (2026).

\bibitem{Thomason97}
R.~W. Thomason, The classification of triangulated subcategories, Compositio Math. {\bf 105} (1997), no.~1, 1--27.

\bibitem{VHCGKRMP24}
L. Van Hirtum, P. De Causmaecker, J. Goemaere, T. Kenter, H. Riebler, M. Lass and C. Plessl, A Computation of the Ninth Dedekind Number Using FPGA Supercomputing, ACM Trans. Reconfigurable Technol. Syst. 17 (2024), no.~3, Article 40, 1--28.

\bibitem{Verasdanis26}
C. Verasdanis, Costratification and actions of tensor-triangulated categories, Doc. Math. (2026), published online first, \url{https://doi.org/10.4171/dm/1069}.

\bibitem{Xu14}
F. Xu, Spectra of tensor triangulated categories over category algebras, Arch. Math. (Basel) {\bf 103} (2014), no.~3, 235--253.

\end{thebibliography}
\end{document}